\documentclass[11pt, reqno]{amsart}

\usepackage{amsthm,amssymb,amstext,amscd,amsfonts,amsbsy,amsxtra,latexsym,amsmath,xcolor,mathrsfs,fancybox,upgreek, soul}
\usepackage[english]{babel}
\usepackage[all,cmtip]{xy}
\usepackage[latin1]{inputenc}
\usepackage{cancel}
\usepackage[draft]{hyperref}
\usepackage{comment,enumitem}
\usepackage{mdframed}
\allowdisplaybreaks

\newtheorem{theorem}{Theorem}
\newtheorem{lemma}[theorem]{Lemma}
\newtheorem{corollary}[theorem]{Corollary}
\newtheorem{proposition}[theorem]{Proposition}

\theoremstyle{definition}
\newtheorem{definition}{Definition}

\theoremstyle{remark}

\newtheorem*{remark}{Remark}
\newtheorem*{example}{Example}
\numberwithin{theorem}{section}
\numberwithin{equation}{section}

\newcommand{\N}{\mathbb{N}}
\newcommand{\Z}{\mathbb{Z}}
\newcommand{\R}{\mathbb{R}}
\newcommand{\C}{\mathbb{C}}

\newcommand{\Q}{\mathbb{Q}}

\newcommand{\SL}{{\text {\rm SL}}}

\newcommand{\sgn}{\operatorname{sgn}}

\newcommand{\im}{\textnormal{Im}}

\makeatletter
\newcommand{\vast}{\bBigg@{3.5}}
\newcommand{\Vast}{\bBigg@{4}}

\makeatother

\renewcommand{\pmod}[1]{\  \,  \left( \mathrm{mod} \,  #1 \right)}

\begin{document}

\title{Vector-valued higher depth quantum modular forms and Higher Mordell integrals.}
\author{Kathrin Bringmann, Jonas Kaszian, Antun Milas}
\maketitle

\begin{abstract} We introduce vector-valued higher depth quantum modular forms and investigate examples coming from characters of representations of vertex algebras. These are expressed 
as rank two false theta series, generalizing unary false theta series studied from several standpoints. We also discuss certain double integrals, which may be viewed as the obstruction
to modularity of the depth two quantum modular forms.
We then find explicit formulas for the double error integrals in the form reminiscent of the classical Mordell integral. 

\end{abstract}


\section{Introduction and statement of results}

\subsection{Mordell integrals and quantum modular forms}
The  Mordell integral is usually defined as a function of two variables 
\begin{equation} \label{zwegers}
 h(z) = h(z;\tau) := \int_{\mathbb{R}}  \frac{\cosh(2 \pi z w)}{\cosh(\pi w)}e^{ \pi i \tau w^2}dw,
\end{equation}
 where $z \in \mathbb{C}$ and  $\tau \in \mathbb{H}$, the complex upper half-plane. Integrals of this form were studied by many mathematicians including Kronecker, Lerch, Ramanujan, Riemann, Siegel, and of course Mordell, who proved that a whole family of integrals reduces to \eqref{zwegers}. From these works it is also known that \eqref{zwegers} occurs as the ``error of modularity'' of Lerch sums which have the shape ($q:=e^{2\pi i\tau}$)
 \begin{equation*}
 \sum_{n\in\Z} \frac{e^{2\pi inz_1} q^{\frac{n^2+n}{2}}}{1-e^{2\pi iz_2}q^n}\quad (z_1,z_2\in\mathbb{C}\setminus\{0\}).
 \end{equation*}
 

The Mordell integral plays an important role in the theory of mock modular forms as shown by Zwegers in his remarkable thesis \cite{Zw1}. Zwegers wrote the integrals in \eqref{zwegers} as Eichler integrals. To be more precise, he showed that, for $a,b\in(-\frac12,\frac12)$ we have
\begin{equation}\label{errormod}
h(a\tau-b) = -e^{-2\pi ia\left(b+\frac12\right)} q^{\frac{a^2}{2}}\int_0^{i\infty} \frac{g_{a+\frac12,b+\frac12}(w)}{\sqrt{-i(w+\tau)}}dw,
\end{equation}
where, for $a,b\in \R$, $g_{a,b}$ is the weight $\frac32$ unary theta function defined by
\begin{equation*}
g_{a,b}(\tau) := \sum_{n\in a+\Z} n e^{2\pi ibn} q^{\frac{n^2}{2}}.
\end{equation*}
Zwegers used \eqref{errormod} to find a completion of Lerch sums, by observing that the error of modularity $h(a\tau-b)$ also appears from integrals which have $-\overline{\tau}$ instead of $0$ as the lower integration limit.




Starting with influential work of Zagier \cite{Za1, Za2}, many authors studied related constructions with Eichler integrals from the perspective of quantum modular forms.
In all of these examples the non-holomorphic part (or ``companion'') is given as 
\begin{equation*}
\int_{-\overline{\tau}}^{i \infty}  \frac{f(w)}{(-i (\tau+w))^{\frac32}} dw,
\end{equation*}
where $f$ is a cuspidal theta function of  weight $\frac{1}{2}$ or $\frac32$. 



The main motivation for this paper is to extend this well-known connection between Eichler and Mordell integrals to higher dimensions by using multiple integrals. 
We provide several explicit examples of this connection in the context of higher depth quantum modular forms introduced by the authors in \cite{BKM} (see also \cite{ABMP}).

\subsection{Vertex algebras and modular invariance of characters}
Another, somewhat unrelated, motivation for this project comes from the study of characters in non-rational conformal field theories, where the modularity (or lack thereof) plays an important  role. 

There is already a growing body of research exploring the modularity of characters beyond  the rational vertex operator algebras 
\cite{CM1,CM2,CR,CMW,STT}.
One general feature of these irrational theories is that they admit typical
modules (labelled by a continuous parameter) and atypical modules
(parametrized by a discrete set which is mostly infinite).  When it comes to modular transformation
properties, the $S$-transformation (with $S:=(\begin{smallmatrix}
0&-1\\1&0
\end{smallmatrix})\in \text{SL}_2(\mathbb Z)$) of a character may produce both typical and atypical characters.
So we expect that
\begin{equation} \label{gen.mod} 
{\rm ch}[{M}]\left(-\frac{1}{\tau}\right) =  \int_{\Omega} S_{M,\nu} {\rm ch}[{M_{\nu}}](\tau) d \nu  +{ \sum_{j \in \mathcal{D}} \alpha_{M,j} {\rm ch}[{M_j}] (\tau)},
\end{equation}
where ${\rm ch}[{M_j}]$ are atypical and ${\rm ch}[{M_\nu}]$  are typical characters. Note that the typical characters often have the form
$\text{ch}[M_\nu]=\frac{q^{\frac{\nu^2}{2}}}{\eta(\tau)^m}$, where $\eta(\tau):=q^{\frac{1}{24}}\prod_{n\geq 1}(1-q^n)$ is Dedekind's $\eta$-function.
Moreover $\Omega$ and $\mathcal{D}$ are domains parametrizing typicals and  atypicals representations, respectively. The reader should exercise caution here -- in some examples 
formulas  like (\ref{gen.mod}) only exist as formal distributions \cite{CR}. Also, as divergent integrals might appear, it is sometimes necessary to introduce additional variables (as in \cite{CM1}) 
to ensure convergence.

This type of generalized modularity is known to hold for characters of certain representations of the affine Lie superalgebras $\widehat{\text{sl}(n|1)}$ for $\mathcal N=2$ and $\mathcal N=4$
superconformal algebras at admissible levels \cite{STT,KW}. In this work atypical characters transform as in  (\ref{gen.mod}) such that the integral part is a Mordell-type integral, which is essentially a consequence of Zwegers' thesis \cite{Zw1}.

In this paper we take a slightly different point of view. As many important (algebraic, analytic and categorical) properties of rational vertex algebras 
are captured by the entries of the $S$-matrix (e.g. quantum dimensions, fusion rules), we expect that the full asymptotic expansion 
of characters and their quotients play a pivotal role for irrational theories. 
More precisely, we believe that these higher coefficients in the asymptotics determine the  ``fusion variety'' via resummation and regularization. The latter was introduced by Creutzig and the third author \cite{CM1}. As shown in \cite{BM}, considerations of asymptotic expansion of characters naturally lead to quantum modular forms.

We now explain, with an example, how the concept of quantum modular forms can be used to obtain (\ref{gen.mod}).
For this we consider the $(1,p)$-singlet algebra and its characters. As explained in \cite{CM1,CMW}, this vertex algebra admits typical and atypical representations. The characters of 
atypical representations $M_{r,s}$ are essentially false theta functions. To be more precise, for $1 \leq s \leq p-1$ and $r \in \mathbb{Z}$, we have
\begin{equation*}
 {\rm ch}[M_{r,s}](\tau) = \frac{1}{\eta (\tau)} \sum_{n\geq 0} \left(
 q^{\frac{1}{4p} \left(2pn -s-pr +2p\right)^2} - q^{\frac{1}{4p} \left(2pn +s-pr+2p \right)^2}\right).
\end{equation*}
Two of the authors have already proven in \cite{BM} that these characters are mixed quantum modular forms, in the sense that  $\mathcal{M}_{r,s}(\tau):=\eta(\tau) {\rm ch}[M_{r,s}](\tau) $ is a weight $\frac{1}{2}$ quantum modular form whose companion (expressed as an Eichler integral) agrees with the original false theta function. 
To be more precise, if we write the asymptotic expansion of the false theta function $f$ as ($\frac hk \in \Q, N \in  \N_0$)
\begin{equation*}
f \left( \frac hk + it \right) = \sum_{n=0}^{N-1} a_{h,k}(n) t^n + O \left(t^N \right),
\end{equation*}
then the Eichler integral $g$ has the expansion
\begin{equation*}
g \left( \frac hk +it \right) = \sum_{n=0}^{N-1} a_{-h,k} (n) (-t)^n + O \left( t^N \right).
\end{equation*}
to all orders when expanding at roots of unity. 
This allows us to transfer modularity questions for characters to better behaved companions $\mathcal{M}^*_{r,s}$ as illustrated in the following 
\begin{example}\label{mainthm} For $1 \leq s \leq p-1$, we have
\begin{equation} \label{qmf.ex}
\mathcal{M}^*_{1,s}\left(\tau\right) -\frac{1}{\sqrt{-i\tau}} \sqrt{\frac{2}{p}}\sum_{k=1}^{p-1} \sin\left(\frac{\pi k(p-s)}{p}\right)\mathcal{M}^*_{1,k}\left(-\frac{1}{\tau}\right)=i\sqrt{2p} \cdot  r_{f_{p-s,p}}(\tau),
\end{equation}
where $r_{j_{j,p}}$ is the theta integral defined in \eqref{rf} for the theta function \eqref{3p}. Note that $r_{f_{j,p}}$ also has  the following representation as Mordell integral
$$r_{f_{j,p}}(\tau)=-\int_{\R} \cot\left(\pi i w+\frac{\pi j}{2p}\right)e^{2\pi i p w^2\tau}dw.$$
As typical characters take the form $\frac{e^{2 \pi p i \tau x^2}}{\eta(\tau)},$
the right-hand side of (\ref{qmf.ex}) can be viewed as the contribution from typical representations  as in (\ref{gen.mod}).
For ${\rm ch}[{M}_{r,s}]$ with $r \neq 1$ a finite $q$-series has to be added to ${\rm ch}[{M}_{1,s}]$ so that the above formula looks slightly more complicated (cf. \cite{BM,CM1}).
\end{example} 

It is desirable to extend the modularity result in \eqref{qmf.ex} to ``higher rank'' $W$-algebras, where false theta functions of higher rank appear as characters \cite{BM2}. 
It was already observed earlier \cite{CM2} that a regularization procedure can be used to derive a more complicated version of (\ref{gen.mod}) involving iterated integrals. As the theory of higher depth quantum modular forms also involves multiple integrals \cite{BKM}, it is tempting to conjecture that these characters combine into vector-valued higher depth quantum modular forms. In this paper, we prove an analogue of \eqref{qmf.ex} this for the simplest nontrivial example coming from an $\frak{sl}_3$ false 
theta function $F(q)$ which was studied recently in \cite{BKM}. 




\subsection{Quantum invariants of knots and $3$-manifolds}
As discussed above quantum modular forms are connected to various aspects of number theory including Maass forms. But originally they appeared in the pioneering work of Zagier (and Zagier-Lawrence) on 
unified quantum invariants of certain $3$-manifolds \cite{Za1,Za2}.  In a recent work of Gukov, Pei, Putrov, and Vafa \cite{GPPV}, the authors proposed new quantum invariants of certain $3$--manifolds expressed as holomorphic $q$-series with 
integral coefficients.  These invariants are in many examples sums of ordinary quantum modular forms. It is expected that more general $3$-manifolds  as well as $\operatorname{SU}(3)$ unified WRT invariants exhibit a more complicated higher depth quantum modularity. 
Understanding their error of modularity certainly requires a solid understanding of higher Mordell integrals.
\subsection{Statement of results}

Define
\begin{equation*}
F(q):= \!\!\!\! \sum_{m_1,m_2 \geq 1 \atop{ m_1 \equiv m_2 \pmod{3}}} \!\!\!\!\!\!\!\! {\rm min}(m_1,m_2) q^{\frac{p}{3}\left(m_1^2+m_2^2+m_1m_2\right)-m_1-m_2+\frac1p} \left(1-q^{m_1}\right)\left(1-q^{m_2}\right)\left(1-q^{m_1 + m_2}\right).
\end{equation*}
In \cite{BKM} the authors decomposed this function as $F(q)=\frac{2}{p}F_1(q^p)+2F_2(q^p)$ with $F_1$ and $F_2$ defined in \eqref{defineF1} and \eqref{defineF2}, respectively.  The function $F_1$ and $F_2$ turn out to have generalized quantum modular properties. This connection goes asymptotically via two-dimensional Eichler integrals. For instance, we showed in \cite{BKM} that $F_1$ asymptotically agrees with an integral of the shape
$$
\int_{-\overline{\tau}}^{i\infty}\int_{w_1}^{i\infty} \frac{f(w_1,w_2)}{\sqrt{-i(w_1+\tau)}\sqrt{-i(w_2+\tau)}}dw_2dw_1
$$
where $f\in S_{\frac{3}{2}}(\chi_1,\Gamma)\otimes S_{\frac32}(\chi_2,\Gamma)$ ($\chi_j$ are certain multipliers and $\Gamma\subset\SL_2(\Z)$).
The modular properties of these integrals follow from the modularity of $f$ which in turn gives quantum modular properties of $F_1$. We call the resulting functions higher depth quantum modular forms. Roughly speaking, depth two quantum modular forms satisfy, in the simplest case, the modular transformation property with $M=\left(\begin{smallmatrix} a & b \\ c & d
\end{smallmatrix}\right) \in \SL_2(\Z)$
\begin{equation}\label{Higherdepth}
f(\tau)-(c\tau+d)^{-k}f(M\tau) \in  \mathcal{Q}_{\kappa} (\Gamma)\mathcal O(R) + \mathcal O(R),
\end{equation}
where  $\mathcal{Q}_{\kappa} (\Gamma)$ is the space of quantum modular forms of weight $\kappa$ and $\mathcal O(R)$ the space of real-analytic functions defined on $R\subset\R$.
Clearly, we can construct examples of depth two simply by multiplying two (depth one) quantum modular forms. In this paper, we prove a vector-valued  version which refines \eqref{Higherdepth}. Roughly speaking, $f(M\tau)$ in \eqref{Higherdepth} is replaced by $\sum_{1\leq\ell\leq N}\chi_{j,\ell}(M)f_{\ell}(M\tau)$ (see Definition 2 for the notation). 

Objects of similar nature -  not invariant under the action of the relevant group but instead they satisfy "higher order" functional equations - have  already appeared in the literature.
{\em Higher-order} modular forms constitute a natural extension of the notion of classical modular form and can be constructed using iterated integrals  \cite{DH,DS}; see also \cite{Ma}. They also appear in connection to percolation theory in mathematical physics \cite{KZ}.

We prove the following theorem.
\begin{theorem}\label{maintheorem1}
	The function $F_1$ is a component of a vector-valued depth two quantum modular form of weight one. The function $F_2$ is a component of a vector-valued quantum modular form of depth two and weight two.
\end{theorem}

We next consider higher-dimensional Mordell integrals. Set, for $\boldsymbol{\alpha}\in\R^2$,
\begin{equation*}
H_{1,\boldsymbol{\alpha}}(\tau) := -\sqrt{3}
\int_0^{i\infty}\int_{w_1}^{i\infty}\frac{\theta_1\left(\boldsymbol{\alpha}; \boldsymbol{w}\right)+\theta_2\left(\boldsymbol{\alpha}; \boldsymbol{w}\right)}{\sqrt{-i\left(w_1+\tau\right)}\sqrt{-i\left(w_2+\tau\right)}}dw_2dw_1
,
\end{equation*}
where the theta functions $\theta_1$ and $\theta_2$  are defined in \eqref{thet1} and \eqref{thet2}, respectively and where throughout the paper we write two-dimensional vectors in bold letters and their components using subscript.

\begin{remark}
Related, but different, iterated integrals were studied by Manin in his work on non-commutative modular symbols \cite{Ma}. 
For further developments see also \cite{Br1,Br2}.
\end{remark}
\begin{remark}
	The function $H_{1,\boldsymbol{\alpha}}$ occurs (basically) as the holomorphic error of modularity (see Proposition \ref{propE1}). The remaining piece is itself already an Eichler integral. 
\end{remark}
Setting
\begin{equation*}
\mathcal F_\alpha(x) := \frac{\sinh(2\pi x)}{\cosh(2\pi x)-\cos(2\pi \alpha)},\qquad \mathcal G_{\alpha}(x):=\frac{\sin(2\pi\alpha)}{\cosh(2\pi x)-\cos(2\pi\alpha)},
\end{equation*} 
we define
\begin{equation*}
g_{1,\boldsymbol{\alpha}}(\boldsymbol{w}):=\begin{cases}
2\mathcal G_{\alpha_1}(w_1)\mathcal G_{\alpha_2}(w_2) -2\mathcal F_{\alpha_1}(w_1)\mathcal F_{\alpha_2}(w_2)&\quad\textnormal{ if } \alpha_1,\alpha_2\notin\Z,\\
-2\mathcal F_0(w_1)\mathcal F_{\alpha_2}(w_2) + \frac{2}{\pi w_1}\mathcal F_{\alpha_2}\left(w_2+\frac{3 w_1}{2}\right)&\quad\textnormal{ if } \alpha_1\in\Z,\,\alpha_2\notin\Z,\\
-2\mathcal F_{\alpha_1}(w_1)\mathcal F_0(w_2) + \frac{2}{\pi w_2}\mathcal F_{\alpha_1}\left(w_1+\frac{w_2}{2}\right)&\quad\textnormal{ if } \alpha_1\notin\Z,\, \alpha_2\in\Z.
\end{cases}
\end{equation*}

\begin{theorem}\label{theoremHigherMordell}	
If $\alpha_1,\alpha_2$ are not both in $\Z$, then we have, with $Q(\boldsymbol{w}):=3w_1^2+w_2^2+3w_1w_2$
\begin{equation*}
H_{1,\boldsymbol{\alpha}}(\tau) = \int_{\R^2} g_{1,\boldsymbol{\alpha}}(\boldsymbol{w}) e^{2\pi i\tau Q(\boldsymbol{w})}dw_1dw_2.
\end{equation*}
In particular, if $\alpha_j \notin \Z$ for $j=1,2$, then we have
	\begin{align*}
	H_{1,\boldsymbol{\alpha}}(\tau)=\int_{\R^2}
	\cot\left(\pi iw_1+\pi\alpha_1\right)
	\cot\left(\pi iw_2+\pi\alpha_2\right)e^{2\pi i\tau Q(\boldsymbol{w})}
	dw_1dw_2.
	\end{align*}

\end{theorem}

\begin{remark}
	 Note that there is a related statement if $\alpha_1,\alpha_2\in\mathbb Z$; however for the purpose of this paper it is not required.
\end{remark}
Similarly, set
\begin{align*}
H_{2,\boldsymbol{\alpha}}(\tau) &:= \frac{\sqrt{3}i}{2\pi}\int_0^{i\infty} \int_{w_1}^{i\infty} \frac{2\theta_3(\boldsymbol{\alpha};\boldsymbol{w}) - \theta_4(\boldsymbol{\alpha}; \boldsymbol{w})}{\sqrt{-i(w_1+\tau)}(-i(w_2+\tau))^{\frac32}}dw_2dw_1 \\
&\qquad+\frac{\sqrt{3}i}{2\pi}\int_0^{i\infty} \int_{w_1}^{i\infty} \frac{\theta_5(\boldsymbol{\alpha}; \boldsymbol{w})}{(-i(w_1+\tau))^{\frac32}\sqrt{-i(w_2+\tau)}}dw_2dw_1,
\end{align*}
where $\theta_3$, $\theta_4$, and $\theta_5$ are theta functions defined in \eqref{thet3}, \eqref{thet4}, and \eqref{thet5}, respectively. The function $ H_{2,\boldsymbol{\alpha}}$ occurs in Proposition \ref{propE1}. 

Define the function $g_{2,\alpha}$ as follows: 
\begin{equation*}
g_{2,\boldsymbol{\alpha}}(\boldsymbol{w}):=\begin{cases}
-2i w_2\left(\mathcal G_{\alpha_1}(w_1)\mathcal F_{\alpha_2}(w_2)+\mathcal F_{\alpha_1}(w_1)\mathcal G_{\alpha_2}(w_2)\right)&\quad\textnormal{ if } \alpha_1\notin\Z,\\
-2i \left(\mathcal F_{0}(w_1)\mathcal G_{\alpha_2}^*(w_2)-\frac{1}{\pi w_1} \mathcal G_{\alpha_2}^*\left(w_2+\frac{3 w_1}{2}\right)\right)&\quad\textnormal{ if } \alpha_1\in\Z,
\end{cases}
\end{equation*}
where  $\mathcal G_\alpha^*(x):=x\mathcal G_{\alpha}(x)$.
\begin{theorem}\label{diffTh}
	We have
	\begin{equation*}
	H_{2,\boldsymbol{\alpha}}(\tau) = \int_{\R^2} g_{2,\boldsymbol{\alpha}}(\boldsymbol{w}) e^{2\pi i\tau Q(\boldsymbol{w})} dw_1dw_2.
	\end{equation*}
\end{theorem}
\subsection{Organization of the paper}
The paper is organized as follows. In Section 2, we recall some basic facts on theta functions, certain (generalized) error functions, quantum modular forms, and higher-dimensional quantum modular forms. Section 3 describes the one-dimensional situation, and Section 4 records our previous results in the two-dimensional case. In Section 5 we develop general vector-valued transformations which we then use for our specific situation. In Section 6 we represent the two theta integrals $H_{1,\boldsymbol{\alpha}}$ and $H_{2,\boldsymbol{\alpha}}$ as double Mordell integrals.

\section*{Acknowledgements} 
The research of the first author is supported by the Alfried Krupp Prize for Young University Teachers of the Krupp foundation
and the research leading to these results receives funding from the European Research Council under the European
Union's Seventh Framework Programme (FP/2007-2013) / ERC Grant agreement n. 335220 - AQSER. The research of the second author is partially supported by the European Research Council under the European
Union's Seventh Framework Programme (FP/2007-2013) / ERC Grant agreement n. 335220 - AQSER.
The research of the third author was partially supported by the 
NSF grant DMS-1601070.
The authors thank Caner Nazaroglu for useful conversations and Josh Males for helpful comments on an earlier version of the paper.
\section{Preliminaries}

\subsection{Theta function transformation}
Define, for $\nu\in\{0,1\}$, $h\in\Z$, $N,A\in\N$ with $A\vert N$ and $N\vert hA$, the theta functions studied by Shimura \cite{Sh}
\begin{equation*}
\Theta_\nu(A,h,N;\tau):=\sum_{\substack{m\in\Z\\m\equiv h\pmod{N}}}m^\nu q^{\frac{Am^2}{2N^2}}.
\end{equation*}
We have the transformation property
\begin{equation}\label{Thetainv}
\Theta_\nu\left(A,h,N;\tau\right)=(-i)^\nu (-i\tau)^{-\frac12-\nu }A^{-\frac12}\sum_{\substack{k\pmod{N} \\ Ak\equiv 0 \pmod{N}}}e\left(\frac{Akh}{N^2}\right)\Theta_\nu \left(A,k,N;-\frac{1}{\tau}\right).
\end{equation}
Also note that if $h_1\equiv h_2\pmod{N}$
\begin{align}\label{Thetneg}
\begin{split}
\Theta_\nu(A,h_1,N;\tau)&=\Theta_\nu(A,h_2,N;\tau), \quad 
\Theta_\nu(A,-h,N;\tau)=(-1)^\nu\Theta_\nu(A,h,N;\tau),\\
 \Theta_\nu(A,N-h,2N;\tau) &=(-1)^\nu \Theta_\nu(A,N+h,2N;\tau).
 \end{split}
\end{align}

\subsection{Special functions}
Following \cite{Zw1}, define for $u\in\R$ 
$$
E(u):=2\int_{0}^{u}e^{-\pi w^2}dw.
$$
We have the representation
\begin{equation*}
E(u)=\sgn(u)\left(1-\frac{1}{\sqrt{\pi }}\Gamma\left(\frac12,\pi u^2\right)\right),
\end{equation*}
where $\Gamma({\alpha},u):=\int_u^\infty e^{-w}w^{{\alpha}-1}dw$ is the {\it incomplete gamma function} and where for $u\in\R$, we let
$$\sgn(u):=\begin{cases}
1 & \text{if $u>0$,}\\
-1 & \text{if $u<0$,}\\
0 & \text{if $u=0$.}
\end{cases}
$$
Moreover, for $u \neq 0$, set
$$
M(u) := \frac{i}{\pi}\int_{\R-iu}\frac{e^{-\pi w^2-2\pi i uw}}{w}dw.
$$
We have
\begin{equation*}\label{defineM}
M(u)=E(u)-\sgn(u).
\end{equation*}

We next turn to two-dimensional analogues, following \cite{ABMP}, however using a slightly different notation.
Setting $\boldsymbol{dw}:=dw_1dw_2$, define $E_2:\R\times\R^2\rightarrow\R$ by 
$$
E_2(\kappa;\boldsymbol{u}):=\int_{\R^2}\sgn\left(w_1\right)\sgn\left(w_2+\kappa w_1\right)e^{-\pi \left(\left(w_1-u_1\right)^2+\left(w_2-u_2\right)^2\right)}\boldsymbol{dw}.
$$
Moreover for $u_2, u_1- \kappa u_2\neq 0$:
\begin{equation}\label{defineM2}
M_2(\kappa; \boldsymbol{u}):=-\frac{1}{\pi^2} \int_{\mathbb{R} - i u_2}  \int_{\mathbb{R}-iu_1} \frac{e^{-\pi w_1^2-\pi w_2^2-2 \pi i (u_1 w_1+u_2 w_2)}}{w_2(w_1-\kappa w_2)}\boldsymbol{dw}.
\end{equation}
Then we have
\begin{equation}\label{M2E2}\begin{split}
M_2\left(\kappa;\boldsymbol{u}\right)&=E_2\left(\kappa;\boldsymbol{u}\right)-\sgn\left(u_2\right)M\left(u_1\right)\\
&\quad-\sgn\left(u_1-\kappa u_2\right)M_1\left(\frac{u_2+\kappa u_1}{\sqrt{1+\kappa^2}}\right)-\sgn\left(u_1\right)\sgn\left(u_2+\kappa u_1\right).
\end{split}\end{equation}
Note that \eqref{M2E2} extends the definition of $M_2$ to $u_2=0$ or $u_1=\kappa u_2$.
\subsection{Vector-valued quantum modular forms}
We next recall vector-valued quantum modular forms for the modular group.
\begin{definition}
An $N$-tuple $\boldsymbol{f}=(f_1,\dots, f_N)$ of functions 
$f_j: \mathbb{Q} \to \C$  for $1\leq j \leq N $ is called a {\it vector-valued quantum modular form of weight $k\in\frac12\mathbb Z$, multiplier $\boldsymbol{\chi}=(\chi_{j,\ell})_{1\leq j,\ell\leq N}$}, if for all $M=\left(\begin{smallmatrix}
	a & b\\ c & d
	\end{smallmatrix}\right)\in\SL_2(\Z)$, the error of modularity
	\begin{equation}
	\label{qtrans}
	f_j(\tau)-(c\tau+d)^{-k} \sum_{1\leq \ell \leq N} \chi_{j,\ell}(M) f_\ell(M\tau)
	\end{equation}
	can be extended to an open subset of $\R$ and is real-analytic there. We denote the vector space of such forms by $\mathcal Q_k(\chi)$. 
\end{definition}
\begin{remark}
	Since the matrices $S:=\left(\begin{smallmatrix}
	0 & -1 \\ 1 & 0
	\end{smallmatrix}\right)$ and $T:=\left(\begin{smallmatrix}
	1 & 1 \\ 0 & 1
	\end{smallmatrix}\right)$ generate $\SL_2(\Z)$, it is enough to check \eqref{qtrans} for these matrices.
\end{remark}

\subsection{Higher depth vector-valued quantum modular forms}
We next introduce vector-valued higher depth quantum modular forms. Note that higher depth quantum modular forms for subgroups of $\SL_2(\Z)$ were considered in \cite{BKM}.

\begin{definition}\label{defgen}
	An $N$-tuple $\boldsymbol{f}=(f_1,\dots,f_N)$ of functions $f_j: \mathbb{Q} \to\C$ with $1\leq j \leq N$ is called a {\it vector-valued quantum modular form of depth $P\in\N$, weight $k\in\frac12\mathbb Z$, multiplier $\boldsymbol{\chi}=(\chi_{j,\ell})_{1\leq j,\ell\leq N}$}, if for all $M=\left(\begin{smallmatrix}
	a & b\\ c & d
	\end{smallmatrix}\right)\in\SL_2(\Z)$,	we have
	\begin{equation*}
	\left( f_j(\tau)-(c\tau+d)^{-k}\sum_{1\leq \ell \leq N} \chi_{\ell,j}(M) f_\ell(M\tau) \right)_{1 \leq j \leq N} \in  \  \sum_m \mathcal Q_{\kappa_m}^{P-1}(\chi_m)\mathcal O(R),
	\end{equation*}
	where $m$ runs through a finite set, $\kappa_m \in \frac12\Z$,  $\chi_{m}$ are rank $N$ multipliers, $\mathcal O(R)$ is the space of real-analytic functions on $R \subset \R$ which contains an open subset of $\R$, $\mathcal Q_k^1(\chi) := \mathcal Q_k(\chi)$, $\mathcal Q_k^0(\chi):=1$, and $\mathcal Q_{k}^{P}(\chi)$ denotes the space of vector-valued forms of weight $k$, depth $P$,  and multiplier $\boldsymbol{\chi}$.
	


\end{definition}

\section{The one-dimensional case}
Recall the classical false theta functions $(1\leq j\le p-1, p\geq 2)$, 
$$
F_{j,p}(\tau):=\sum_{\substack{n\in\Z \\ n\equiv j \pmod{2p}}}\sgn(n) q^{\frac{n^2}{4p}}.
$$
The following theorem is shown in \cite{BM, CMW} (note that here we renormalized in comparison to \cite{BKM})
\begin{theorem} 
	The functions $F_{j,p}:\mathbb H\rightarrow\C$ $(1\leq j\leq p-1)$ form a vector-valued quantum modular form.
\end{theorem}
\begin{proof} (Sketch)
Define the {\it non-holomorphic Eichler integral}
\begin{equation*}
F^*_{j,p}(\tau):=\frac{1}{\sqrt{\pi}}\sum_{\substack{n\in\Z \\ n\equiv j\pmod{2p}}}\sgn(n)\Gamma\left(\frac12,\frac{\pi n^2 v}{p}\right)q^{-\frac{n^2}{4p}}.
\end{equation*}
Note that $F_{j,p}(it+\frac{h}{k})$ and $F^*_{j,p}(it-\frac{h}{k})$ agree asymptotically to infinite order. That is, if we write
\[
F_{j,p}\left(it+\frac{h}{k}\right)\sim \sum_{m\geq 0}a_{h,k}(m)t^m\quad\left(t\rightarrow 0^+\right),
\]
then
\[
F^*_{j,p}\left(it-\frac{h}{k}\right)\sim \sum_{m\geq0}a_{h,k}(m)(-t)^m\quad\left(t\rightarrow0^+\right).
\]
One may then show that 
\begin{equation*}
F_{j,p}^\ast(\tau)=- i\sqrt{2p} \cdot I_{f_{j,p}}(\tau),
\end{equation*}
where 
\begin{equation}
\label{3p}
f_{j,p}(z):=\frac{1}{2p}\sum_{\substack{n\in\Z\\ n\equiv j\pmod{2p} }}nq^{\frac{n^2}{4p}}
\end{equation}
and for a holomorphic modular form from $f$ of weight $k$, the \emph{non-holomorphic Eichler integral} is
\[
I_f(\tau):=\int_{-\overline \tau}^{i\infty}\frac{f(w)}{(-i(\tau+w))^{2-k}}dw.
\]
Using \eqref{Thetainv}, one can prove that 
\begin{equation*}
f_{j,p}\left(\tau\right)= \sqrt{\frac{2}{p}} (- i \tau)^{-\frac32}\sum_{k=1}^{p-1} {\rm sin}\left(\frac{\pi k j}{p}\right) f_{k,p}\left(-\frac{1}{\tau}\right),
\end{equation*}
correcting a sign-error in \cite{CMW}. From this one may conclude that
\begin{align*}
 F^*_{j,p}\left(\tau\right) -\frac{1}{\sqrt{-i\tau}} \sqrt{\frac{2}{p}}\sum_{k=1}^{p-1} \sin\left(\frac{\pi kj}{p}\right)F^*_{k,p}\left(-\frac{1}{\tau}\right)=i\sqrt{2p}\cdot r_{f_{j,p}}(\tau),
\end{align*}
where, for $f$ a holomorphic modular form of weight $k$,
\begin{equation}\label{rf}
r_f(\tau):=\int_0^{i\infty}\frac{f(w)}{(-i(w+\tau))^{2-k}}dw.
\end{equation}
The claim now follows since $r_{f_{j,p}}$ is real-analytic on $\R$.
\end{proof}
The next lemma writes the ``error to modularity'' as an Eichler integral.
Following the approach of Zwegers \cite{Zw1} and using trigonometric identities, one finds the following.
\begin{lemma}
	We have
\begin{equation*}
\begin{aligned}
-i\sqrt{2p}\cdot r_{f_{j,p}}(\tau)&=\int_{\R} \cot\left(\pi i w+\frac{\pi j}{2p}\right)e^{2\pi i p\tau w^2}dw
\\
&= \sin\left(\frac{\pi j}{p}\right) \frac12 \int_{\R} \frac{e^{2\pi i p\tau w^2}}{\sinh\left(\pi w+\frac{\pi ij}{2p}\right) \sinh\left(\pi w-\frac{\pi ij}{2p}\right)}dw.
\end{aligned}
\end{equation*}
\end{lemma}

\section{Previous results in the two-dimensional case}

In this section, we recall the results from \cite{BKM}. In that paper the following decomposition was shown
\begin{equation*}
F(q)=\frac2pF_1\left(q^p\right)+2F_2\left(q^p\right)
\end{equation*}
with
\begin{align}\label{defineF1}
F_1(q):=\sum_{\boldsymbol{\alpha}\in\mathscr S}\varepsilon(\boldsymbol{\alpha})\sum_{\boldsymbol{n}\in\boldsymbol{\alpha} +\N_0^2}q^{Q(\boldsymbol{n})}+\frac12 \sum_{m\in\Z}\sgn\left(m+\frac1p\right)q^{\left(m+\frac1p\right)^2},
\end{align}
where 
\begin{align*}
\mathscr{S}:=&\left\{\left(1-\frac{1}{p},\frac{2}{p}\right),\left(\frac{1}{p},1-\frac{2}{p}\right),\left(1,\frac{1}{p}\right)\left(0,1-\frac{1}{p}\right),\left(\frac{1}{p},1-\frac{1}{p}\right),\left(1-\frac{1}{p},\frac{1}{p}\right)\right\},
\end{align*}
and for $\boldsymbol{\alpha}\pmod{\Z^2}$, we set
\begin{align*}
\varepsilon(\boldsymbol{\alpha}):=&\begin{cases} -2 \qquad&\text{if } \boldsymbol{\alpha} \in\left\{\left(1-\frac{1}{p},\frac{2}{p}\right),\left(\frac{1}{p},1-\frac{2}{p}\right)\right\},\\ 1 &\text{otherwise}.\end{cases}
\end{align*}
Moreover
\begin{align}\label{defineF2}
F_2(q)&:=\sum_{\boldsymbol{\alpha}\in\mathscr{S}}\eta(\boldsymbol{\alpha})\sum_{\boldsymbol{n}\in\boldsymbol{\alpha}+\N_0^2} n_2 q^{Q(\boldsymbol{n})}-\frac12\sum_{m\in\Z}\left|m+\frac1p\right|q^{\left(m+\frac1p\right)^2},
\end{align}
where for $\boldsymbol{\alpha}\pmod{\Z^2}$, we let
\begin{align*}
\eta(\boldsymbol{\alpha})&:=
\begin{cases}
1&\quad\text{ if } \boldsymbol{\alpha} \in\left\{\left(1-\frac1p, \frac2p\right), \left(0, 1-\frac1p\right), \left(\frac1p, 1-\frac1p\right)\right\},\\
-1&\quad\text{ otherwise.}
\end{cases}
\end{align*}
In \cite{BKM} the following theorem was shown. 
\begin{theorem}
	For $p\geq 2$, the functions 
	$F_1$ and $F_2$ are quantum modular forms of depth two with quantum set $\Q$ and of weight one and weight two, respectively. 
\end{theorem}
\begin{proof}[Sketch of proof]
Using the Euler-Maclaurin summation formula, it was shown in \cite{BKM} that the higher rank false theta functions asymptotically equal double Eichler integrals. To be more precise, write
\[
F_1\left(e^{2\pi i\frac{h}{k}-t}\right)\sim\sum_{m\geq 0} A_{h, k}(m) t^m\quad \left(t\to 0^+\right).
\]
In \cite{BKM}, we proved that we have, for $h,k\in\Z$ with $k>0$ and $\gcd(h,k)=1$,
\begin{equation}\label{AsE1}
\mathbb{E}_1\left(\frac{it}{2\pi}-\frac{h}{k}\right){\sim } \sum_{m\geq 0} A_{h, k}(m) (-t)^m\quad \left(t\to 0^+\right).
\end{equation}
Here the double Eichler integral $\mathbb{E}_1$ is given as follows: Define for $\boldsymbol{\alpha} \in \mathscr{S}^\ast:=\{(1-\frac1p,\frac2p),$\\$(0,1-\frac1p),(\frac1p,1-\frac1p)\}$
\begin{align*}
\mathcal E_{1,\boldsymbol{\alpha}}(\tau):=-\frac{\sqrt{3}}{4}
\int_{-\overline{\tau}}^{i\infty}
\int_{w_1}^{i\infty}
\frac{\theta_1(\boldsymbol{\alpha};\boldsymbol{w})+\theta_2(\boldsymbol{\alpha};\boldsymbol{w})}{\sqrt{-i(w_1+\tau)}\sqrt{-i(w_2+\tau)}}
dw_2 dw_1
\end{align*}
with
\begin{align}\label{thet1}
\theta_1(\boldsymbol{\alpha};\boldsymbol{w})&:=\sum_{n\in\alpha+\Z^2}(2n_1+n_2)n_2
e^{\frac{3\pi i}{2}(2n_1+n_2)^2w_1
	+ \frac{\pi in_2^2w_2}{2}},
\\
\label{thet2}
\theta_2(\boldsymbol{\alpha};\boldsymbol{w})&:=\sum_{\boldsymbol{n}\in \boldsymbol{\alpha}+\Z^2}(3n_1+2n_2)n_1
e^{\frac{\pi i}{2}(3n_1+2n_2)^2w_1
	+ \frac{3\pi in_1^2w_2}{2}}.
\end{align}
Then set
\begin{align}\label{Ep1}
\mathcal{E}_1(\tau)&:= \sum_{\boldsymbol{\alpha}\in \mathscr{S}^\ast}\varepsilon(\boldsymbol{\alpha})\mathcal{E}_{1,\boldsymbol{\alpha}}(p\tau),\qquad \mathbb E_1(\tau):=\mathcal E_1\left(\frac{\tau}{p}\right).
\end{align}
The double Eichler integral $\mathcal E_1$ satisfies modular transformation properties. To be more precise, we have, for  $M=\left(\begin{smallmatrix}a&b\\c&d\end{smallmatrix}\right)\in\Gamma_p$ (some congruence subgroup of $\text{SL}_2(\mathbb Z)$),
	$$
	\mathcal E_1(\tau) - \left(\frac{-3}{d}\right) (c\tau+d)^{-1}\mathcal E_1(M\tau) = \sum_{j=1}^{2} \left(r_{f_j,g_j,\frac{d}{c}}(\tau)+I_{f_j}(\tau) r_{g_j,\frac{d}{c}}(\tau)\right),
	$$
	where $\left(\frac{\,\cdot\,}{\,\cdot\,}\right)$ is the extended Jacobi symbol, $f_j,g_j$ are cusp forms of weight $\frac32$ (with some multiplier), and for holomorphic modular forms $f_1$ and $f_2$ of weights $\kappa_1$ and $\kappa_2$, respectively, we set
	\begin{align*}
	r_{f_1,f_2,\frac{d}{c}}(\tau) &:= \int_{\frac{d}{c}}^{i\infty} \int_{w_1}^{\frac{d}{c}} \frac{f_1(w_1) f_2(w_2)}{(-i(w_1+\tau))^{2-\kappa_1} (-i(w_2+\tau))^{2-\kappa_2}} dw_2 dw_1,\\
	r_{f_1, \frac dc}(\tau)&:=\int_{\frac dc}^{i\infty}\frac{f_1(w)}{\left(-i(w+\tau)\right)^{2-\kappa_1}}dw.
	\end{align*}

The situation is similar for $F_2$. To be more precise, writing
$$
F_2\left(e^{2\pi i\frac{h}{k}-t}\right)\sim \sum_{m\geq 0} B_{h,k}(m) t^m \qquad \left(t\to 0^+\right),
$$
we proved in \cite{BKM} that we have, for $h,k\in\Z$ with $k>0$ and $\gcd(h,k)=1$,
\begin{equation}\label{AsE2}
\mathbb E_2\left(\frac{it}{2\pi}-\frac{h}{k}\right) \sim  \sum_{m\geq 0} B_{h,k}(m) (-t)^{m}\qquad \left(t\to 0^+\right).
\end{equation}
Here the Eichler integral $\mathbb E_2$ is given as follows: Define for $\boldsymbol{\alpha}\in\mathscr S^\ast$
\begin{align*}
\mathcal E_{2,\boldsymbol{\alpha}}(\tau):=&\frac{\sqrt{3}}{8\pi} \int_{-\overline{\tau}}^{i\infty} \int_{w_1}^{i\infty} \frac{2\theta_3(\boldsymbol{\alpha}; \boldsymbol{w}) -\theta_4(\boldsymbol{\alpha}; \boldsymbol{w})}{\sqrt{-i(w_1+\tau)}(-i(w_2+\tau))^{\frac32}}dw_2dw_1 \\
&+\frac{\sqrt{3}}{8\pi} \int_{-\overline{\tau}}^{i\infty} \int_{w_1}^{i\infty} \frac{\theta_5(\boldsymbol{\alpha};\boldsymbol{w})}{(-i(w_1+\tau))^{\frac32}\sqrt{-i(w_2+\tau)}} dw_2 dw_1
\end{align*}
with
\begin{align}\label{thet3}
\theta_3(\boldsymbol{\alpha}; \boldsymbol{w})&:= \sum_{n\in\boldsymbol{\alpha}+\Z^2}(2n_1+n_2) e^{\frac{3\pi i}{2}(2n_1+n_2)^2w_1+\frac{\pi i n_2^2 w_2}{2}},\\
\label{thet4}
\theta_4(\boldsymbol{\alpha}; \boldsymbol{w})&:= \sum_{\boldsymbol{n}\in\boldsymbol{\alpha}+\Z^2}(3n_1+2n_2) e^{\frac{\pi i}{2}(3n_1+2n_2)^2w_1+\frac{3\pi i n_1^2 w_2}{2}},\\
\label{thet5}
\theta_5(\boldsymbol{\alpha}; \boldsymbol{w})&:= \sum_{\boldsymbol{n}\in\boldsymbol{\alpha}+\Z^2}n_1 e^{\frac{\pi i}{2}(3n_1+2n_2)^2w_1+\frac{3\pi i n_1^2 w_2}{2}}.
\end{align}
We then set
\begin{equation}\label{Ep2}
\mathcal E_2(\tau) := \sum_{\boldsymbol{\alpha}\in\mathscr S^\ast}\mathcal E_{2,\boldsymbol{\alpha}}(p\tau),\qquad \mathbb E_2(\tau) := \mathcal E_2\left(\frac{\tau}{p}\right).
\end{equation}
Again one can show transformations for $\mathcal E_2$. Namely
	for $M\in\Gamma_p$, one has
	$$
	\mathcal E_2(\tau) - \left(\frac{3}{d}\right) (c\tau+d)^{-2} \mathcal E_2(M\tau) = \sum_{j=1}^{4} \left(r_{f_j,g_j,\frac{d}{c}}(\tau)+I_{f_j}(\tau)r_{g_j,\frac{d}{c}}(\tau)\right),
	$$
	with $f_j$ and $g_j$  holomorphic modular forms of weight $\frac12$ or cusp forms of weight $\frac32$, respectively.
\end{proof}
\section{Higher depth Vector-valued transformations}

\subsection{General double Eichler integrals}
We first describe the general situation. For this assume that $f_{j}, g_{\ell}$ $(1\leq j\leq N,\ 1\leq \ell\leq M)$ are components of vector-valued modular forms and in particular transform as (with $\kappa_1,\kappa_2\in\frac12+\N_0$)  
\begin{equation}\label{fgtran}
f_j\left(-\frac1\tau\right) = (-i\tau)^{\kappa_1} \sum_{1\le k\le N} \chi_{j,k}\: f_k(\tau),
\quad
g_\ell\left(-\frac1\tau\right)= (-i\tau)^{\kappa_2} \sum_{1\le m \le M} \psi_{\ell,m}\: g_m(\tau).
\end{equation}
Following \cite{BKM}, define the {\it double Eichler integral}
\begin{align*}
I_{f_j,g_\ell}(\tau):=&\ \int_{-\overline{\tau}}^{i\infty}\int_{w_1}^{i\infty} \frac{f_j(w_1)g_\ell(w_2)}{(-i(w_1+\tau))^{2-\kappa_1}(-i(w_2+\tau))^{2-\kappa_2}}dw_2dw_1.
\end{align*}

We prove the following transformation.

\begin{lemma}\label{Itra}
We have the following two transformations
\begin{align}\label{Tshift}
&I_{f_j,g_\ell}(\tau)-I_{f_j|T,g_\ell|T}(\tau+1)=0,\\
	\label{Stran}
	&I_{f_j,g_\ell}(\tau)-(-i\tau)^{\kappa_1+\kappa_2-4}\sum_{\substack{1\leq k\leq N\\ 1\leq m\leq M}}\chi_{j,k}\psi_{\ell,m}I_{f_k,g_m}\left(-\frac{1}{\tau}\right)\\\notag 
	&\qquad=\int_0^{i\infty}\int_{w_1}^{i\infty}\frac{f_j(w_1)g_{\ell}(w_2)}{(-i(w_1+\tau))^{2-\kappa_1}(-i(w_2+\tau))^{2-\kappa_2}}dw_2dw_1+I_{f_j}(\tau)r_{g_\ell}(\tau)-r_{f_j}(\tau) r_{g_\ell}(\tau),
	\end{align}
	where $|_{\kappa}$ denotes the usual weight $k$ slash operator.
\end{lemma}
\begin{proof}
The transformation \eqref{Tshift} is clear. To show \eqref{Stran}, we first compute, using \eqref{fgtran},
\begin{equation*}
(-i\tau)^{\kappa_1+\kappa_2-4}\sum_{\substack{1\leq k\leq N\\1\leq m\leq M}}\chi_{j,k}\psi_{\ell,m}I_{f_k,g_m}\left(-\frac{1}{\tau}\right) = \int_{-\overline{\tau}}^{0}\int_{w_1}^0 \frac{f_j(w_1)g_\ell(w_2)}{(-i(w_1+\tau))^{2-\kappa_1}(-i(w_2+\tau))^{2-\kappa_2}}dw_2dw_1.
\end{equation*}
Employing the splitting 
\begin{equation*}
\int_{-\overline{\tau}}^{0} \int_{w_1}^0 = \int_{-\overline{\tau}}^{i\infty} \int_{w_1}^{i\infty} + \int_0^{i\infty} \int_0^{i\infty}-\int_0^{i\infty}\int_{w_1}^{i\infty} - \int_{-\overline{\tau}}^{i\infty} \int_0^{i\infty}
\end{equation*}
then directly gives the claim.
\end{proof}

\subsection{The function $\mathcal E_1$}
We first rewrite $\mathcal E_1$. For this define, for $k_1,k_2\in\Z$ with $k_1\equiv k_2\pmod{2}$,  
\begin{align*}
J_{\boldsymbol{k}}(\tau)&:=\sum_{\delta\in\{0,1\}}I_{(k_1+\delta p,k_2+3\delta p)}(\tau)\quad\text{with}\quad
I_{\boldsymbol{k}}(\tau) := -\frac{\sqrt{3}}{4p}I_{\Theta_1(2p,k_1,2p;\cdot),\Theta_1(6p,k_2,6p,\cdot)}(\tau),\\
r_{\boldsymbol{k}}(\tau)&:=\int_0^{i\infty}\int_{w_1}^{i\infty}\frac{\Theta_1(2p,k_1,2p;w_1)\Theta_1(6p,k_2,6p;w_2)}{\sqrt{-i(w_1+\tau)}\sqrt{-i(w_2+\tau)}}dw_2dw_1.
\end{align*}
We have the following transformation properties.
\begin{proposition}\label{propJ}
	We have, for $\ell_1\equiv \ell_2\pmod{2}$, 
	\begin{align*}
	J_{\boldsymbol{\ell}}(\tau) &= -\frac{1}{\sqrt{3}p(-i\tau)}\sum_{\substack{k_1\pmod{p} \\ k_2\pmod{6p} \\ k_1\equiv k_2\pmod{2}}} \zeta_{2p}^{k_1\ell_1}\zeta_{6p}^{k_2\ell_2} J_{\boldsymbol{k}}\left(-\frac{1}{\tau}\right)- \frac{\sqrt{3}}{4p} \sum_{\delta\in\{0,1\}} r_{(k_1+p\delta,k_2+3p\delta)}(\tau)\\
	&\quad -\frac{\sqrt{3}}{4p}\sum_{\delta\in\{0,1\}}\Bigg( I_{\Theta_1(2p,\ell_1+p\delta,2p;\,\cdot\,)}(\tau)-r_{\Theta_1(2p,\ell_1+p\delta,2p;\,\cdot\,)}(\tau)\Bigg)r_{\Theta_1(6p,\ell_2+3p\delta,6p;\,\cdot\,)}(\tau),
	\end{align*}
	where $\zeta_j:=e^{\frac{2\pi i}{j}}$.
	\end{proposition}
	\begin{proof}
		Using \eqref{Thetainv} gives  
		\begin{align}\label{ThetaB}
		\Theta_1\left(2p,a,2p;-\frac{1}{\tau}\right)&=-i(-i\tau)^\frac32 (2p)^{-\frac12}\sum_{k\pmod{2p}}\zeta_{2p}^{ka}\Theta_1(2p,k,2p;\tau),\\
		\notag
		\Theta_1\left(6p,a,6p;-\frac{1}{\tau}\right)&=-i(-i\tau)^\frac32(6p)^{-\frac12}\sum_{k\pmod{6p}}\zeta_{6p}^{ka}\Theta_1(6p,k,6p;\tau).
		\end{align}
		Thus by Lemma \ref{Itra}, we obtain that $J_{\boldsymbol{\ell}}(\tau)$ equals
		\begin{equation*}
		\begin{aligned}
		& -\frac{1}{2\sqrt{3}p(-i\tau)} \sum_{\delta\in\{0,1\}}  \sum_{k_1\pmod{2p}\atop{k_2\pmod{6p}}} \zeta_{2p}^{k_1(\ell_1+p\delta)}\zeta_{6p}^{k_2(\ell_2+3p\delta)} I_{\boldsymbol{k}}\left(-\frac{1}{\tau}\right)-\frac{\sqrt{3}}{4p}\sum_{\delta\in\{0,1\}}r_{(\ell_1+p\delta,\ell_2+3p\delta)}(\tau)\\
		&\qquad-\frac{\sqrt{3}}{4p}\sum_{\delta\in\{0,1\}}\Bigg( I_{\Theta_1(2p,\ell_1+p\delta,2p;\cdot)}(\tau)-r_{\Theta_1(2p,\ell_1+p\delta,2p;\cdot)}(\tau)\Bigg)r_{\Theta_1(6p,\ell_2+3p\delta,6p;\cdot)}(\tau).
		\end{aligned}
		\end{equation*}

		To prove the proposition, we are left to simplify the first term. For this, we write
		\begin{align*}	
		\sum_{\delta\in\{0, 1\}}\sum\limits_{k_1\pmod{2p}\atop{k_2\pmod{6p}}}(-1)^{\delta\left(k_1+k_2\right)}\zeta_{2p}^{\ell_1k_1}\zeta_{6p}^{\ell_2k_2} I_{\boldsymbol{k}}\left(-\frac{1}{\tau}\right)
		=2\sum\limits_{{k_1\pmod{2p}\atop{k_2\pmod{6p}}}\atop{k_1\equiv k_2\pmod{2}}}\zeta_{2p}^{\ell_1k_1}\zeta_{6p}^{\ell_2k_2} I_{\boldsymbol{k}}\left(-\frac{1}{\tau}\right).
		\end{align*}
		Making the change of variables $k_1\mapsto k_1+p\delta$, $k_2\mapsto k_2+3p\delta$ yields that this equals
		\begin{multline*}
		2\sum_{{k_1\pmod{p}\atop{k_2\pmod{6p}}}\atop{k_1\equiv k_2\pmod{2}}} \sum_{\delta\in\{0,1\}} \zeta_{2p}^{(k_1+p\delta)\ell_1}\zeta_{6p}^{(k_2+3p\delta)\ell_2} I_{(k_1+p\delta,k_2+3p\delta)}\left(-\frac{1}{\tau}\right)\\
		=2\sum_{{k_1\pmod{2p}\atop{k_2\pmod{6p}}}\atop{k_1\equiv k_2\pmod{2}}}\zeta_{2p}^{\ell_1k_1}\zeta_{6p}^{\ell_2k_2}J_{\boldsymbol{k}}\left(-\frac{1}{\tau}\right). \qedhere
		\end{multline*}
	\end{proof}
	To find transformation properties to use for $\mathcal E_1$, we write it as a $J$-function.
\begin{lemma}
	We have
	\begin{align}\label{EI}
	\mathcal E_1(\tau)=J_{(1,3)}(\tau).
	\end{align}
\end{lemma}
\begin{proof}
	As in the proof of Proposition 5.2 of \cite{BKM} we see that
	\begin{align*}
	\sum_{\boldsymbol{\alpha}\in\mathscr S^*} \varepsilon\left(\boldsymbol{\alpha}\right) \theta_1(\boldsymbol{\alpha};\boldsymbol{w})=\frac{1}{p^2}\sum_{\boldsymbol{A}\in\mathcal A}\varepsilon_1\left(\boldsymbol{A}\right)\Theta_1\left(2p,A_1,2p;\frac{3w_1}{p}\right)\Theta_1\left(2p,A_2,2p;\frac{w_2}{p}\right)
	\end{align*}
	with
	\begin{align*}
	\mathcal A:=\!\left\{\left(0,2\right),\left(p,p+2\right),\left(p-1,p-1\right),\left(-1,-1\right),\left(p+1,p-1\right),\left(1,-1\right)\right\}\!, \varepsilon_1(\boldsymbol{A}):= \varepsilon\left(\frac{A_1-A_2}{2p},
	\frac{A_2}{p}\right).
	\end{align*}	
	Using \eqref{Thetneg}, it is not hard to prove that this sum vanishes.
	
	Similarly
	\begin{align}\label{sumB}
	\sum_{\boldsymbol{\alpha}\in\mathscr S^*}
	\varepsilon\left(\boldsymbol{\alpha}\right)
	\theta_2(\boldsymbol{\alpha};\boldsymbol{w})=\frac{1}{p^2}\sum_{\boldsymbol{B}\in \mathcal B}\varepsilon_2\left(\boldsymbol{B}\right)
	\Theta_1\left(2p,B_1,2p;\frac{w_1}{p}\right)\Theta_1\left(2p,B_2,2p;\frac{3w_2}{p}\right)
	\end{align}
	with
	\begin{align*}
	\mathcal{B}&:=\!\left\{(p+1,p-1),(1,-1),(p+2,p),(2,0),(1,1),(p+1,p+1)\right\},\
	\varepsilon_2(\boldsymbol{B}):=\varepsilon\left( \frac{B_2-3B_1}{2p},\frac{B_1}{p}\right).
	\end{align*}
	Using again \eqref{Thetneg} and $\Theta_1(2p,h,2p;3\tau)=\frac13\Theta_1(6p,3h,6p;\tau)$, one obtains that \eqref{sumB} equals
	\begin{equation*}
	\frac1{p^2}\sum_{\delta\in\{0,1\}} \Theta_1\left(2p,1+\delta p,2p;\frac{w_1}{p}\right)\Theta_1\left(6p,3+3\delta p,6p;\frac{w_2}{p}\right).
	\end{equation*}
	This yields the claim by \eqref{Ep1}.
\end{proof}
Proposition \ref{propJ} then implies the following transformation for $\mathcal E_1$.
\begin{corollary}\label{propE1}
	We have 
	\begin{align*}
	\mathcal E_1(\tau)&=-\frac{1}{\sqrt{3}p(-i\tau)}\sum_{\substack{k_1\pmod{p}\\k_2\pmod{6p}\\k_1\equiv k_2\pmod{2}}}\zeta_{2p}^{k_1+k_2}J_{\boldsymbol{k}}\left(-\frac{1}{\tau}\right)+\frac{1}{4}\sum_{\boldsymbol{\alpha}\in\mathscr S^*}\varepsilon(\boldsymbol{\alpha})H_{1,\boldsymbol{\alpha}}(\tau)\\
	&\quad -\frac{\sqrt{3}}{4p}\sum_{\delta\in\{0,1\}}\left(I_{\Theta_1(2p,1+p\delta,2p; .)}(\tau)-r_{\Theta_1(2p,1+p\delta,2p; .)}(\tau)\right)r_{\Theta_1(6p,3+3p\delta,6p;.)}(\tau).
	\end{align*}
\end{corollary}
\begin{proof}
	We use Proposition \ref{propJ} with $\ell_1=1$ and $\ell_2=3$ and reversing the calculation used to show \eqref{EI}, we obtain that the second term equals
	$
	\frac14\sum_{\boldsymbol{\alpha}\in\mathscr S^*}\varepsilon(\boldsymbol{\alpha})H_{1,\boldsymbol{\alpha}}(\tau)
	$.
	
\end{proof}

\subsection{The function $\mathcal E_2$}
We proceed in the same way as for $\mathcal E_1$.
To rewrite $\mathcal E_2$, defined in \eqref{Ep2}, we set, for $k_1\equiv k_2\pmod{2}$,
\[
\mathcal K_{\boldsymbol{k}}(\tau):=2\mathcal J_{\boldsymbol{k}}(\tau)+\mathcal J_{\left(\frac{k_1+k_2}{2},\frac{k_2-3k_1}{2}\right)}(\tau),
\]
where (note that we changed the normalization in comparison to \cite{BKM})
\begin{align*}
\mathcal J_{\boldsymbol{k}}(\tau):=\sum_{\delta\in\{0,1\}}\mathcal I_{(k_1+p\delta,k_2+3p\delta)}(\tau),
\quad\text{with}\quad
\mathcal I_{\boldsymbol{k}}(\tau):=-\frac{\sqrt{3}}{8\pi}I_{\Theta_1(2p,k_1,2p;\cdot),\Theta_0(6p,k_2,6p;\cdot)}(\tau).
\end{align*}
Moreover set 
\[
R_{\boldsymbol{k}}(\tau):=\int_0^{i\infty}\int_{w_1}^{i\infty}\frac{\Theta_1(2p,k_1,2p;w_1)\Theta_0(6p,k_2,6p;w_2)}{\sqrt{-i(w_1+\tau)}(-i(w_2+\tau))^\frac32}dw_2dw_1.
\]
We have the following transformation law for the function $\mathcal K_\ell$.
\begin{proposition}\label{propGenTrans}
	We have, for $\ell_1\equiv \ell_2 \pmod{2}$,
	\begin{align*}
	&\mathcal K_{\boldsymbol{\ell}}(\tau)=\frac{i}{2\sqrt{3}p}\sum_{\substack{ k_1\pmod{p} \\ k_2\pmod{6p} \\ k_1\equiv k_2\pmod{2}}} \zeta_{2p}^{k_1\ell_1}\zeta_{6p}^{k_2\ell_2}\mathcal K_{\boldsymbol{k}}\left(-\frac{1}{\tau}\right)\\
	&\qquad\quad-\frac{\sqrt{3}}{8\pi}\sum_{\delta\in\{0,1\}}\left(2R_{(k_1+p\delta,k_2+3p\delta)}(\tau)+R_{\left(\frac{k_1+k_2}{2}+p\delta,\frac{k_2-3k_1}{2}+3p\delta\right)}(\tau)\right)\\
	&\qquad\quad-\frac{\sqrt{3}}{8\pi}\sum_{\delta\in\{0,1\}}\Bigg(2\left(I_{\Theta_1\left(2p,\ell_1+p\delta,2p;.\right)}(\tau)-r_{\Theta_1(2p,\ell_1+p\delta,2p;.)}(\tau)\right)r_{\Theta_0\left(6p,\ell_2+3p\delta,6p;.\right)}(\tau)\\
	&\qquad\quad+\left(I_{\Theta_1\left(2p,\frac{\ell_1+\ell_2}{2}+p\delta,2p;.\right)}(\tau)-r_{\Theta_1\left(2p,\frac{\ell_1+\ell_2}{2}+p\delta,2p;.\right)}(\tau)\right)r_{\Theta_0\left(6p,\frac{\ell_2-3\ell_1}{2}+3p\delta,6p;.\right)}(\tau)\Bigg).
	\end{align*}
\end{proposition}

\begin{proof}
	Using \eqref{ThetaB} and
	\begin{align*}
	\Theta_0\left(6p,a,6p;-\frac{1}{\tau}\right)&= (-i\tau)^{\frac12} \frac{1}{\sqrt{6p}} \sum_{k \pmod{6p}} \zeta_{6p}^{ka} \Theta_0(6p,k,6p;\tau),
	\end{align*}
	Proposition \ref{prop2} gives that $\mathcal K_{\ell_1,\ell_2}(\tau)$ equals 
	\begin{align*}
	&\sum_{\substack{k_1\pmod{2p}\\k_2\pmod{6p}}}\left(2\zeta_{2p}^{k_1(\ell_1+p\delta)}\zeta_{6p}^{k_2(\ell_2+3p\delta)}+2\zeta_{2p}^{k_1\left(\frac{\ell_1+\ell_2}{2}+p\delta\right)}\zeta_{6p}^{k_2\left(\frac{\ell_2-3\ell_1}{2}+3p\delta\right)}\right)\frac{i}{16p\pi(-i\tau )^2}
	\mathcal I_{\boldsymbol{k}}\left(-\frac{1}{\tau}\right)\\
	&-\frac{\sqrt{3}}{8\pi}\sum_{\delta\in\{0,1\}}\left(2R_{(\ell_1+p\delta,\ell_2+3p\delta)}(\tau)+R_{\left(\frac{\ell_1+\ell_2}{2}+p\delta,\frac{\ell_2-3\ell_1}{2}+3p\delta\right)}(\tau)\right)\\
	&-\frac{\sqrt{3}}{8\pi}\sum_{\delta\in\{0,1\}}\Bigg(2\left(I_{\Theta_1(2p,\ell_1+p\delta,2p;.)}(\tau)-r_{\Theta_1(2p,\ell_1+p\delta,2p;.)}(\tau)\right)r_{\Theta_0(6p,\ell_2+3p\delta,6p;\cdot)}(\tau)\\
	&\qquad\qquad\qquad+\left(I_{\Theta_1\left(2p,\frac{\ell_1+\ell_2}{2}+p\delta,2p;.\right)}(\tau)-r_{\Theta_1\left(2p,\frac{\ell_1+\ell_2}{2}+p\delta,2p;.\right)}(\tau)\right)r_{\Theta_0(6p,\ell_2+3p\delta,6p;\cdot)}(\tau)\Bigg).
	\end{align*}
	
	We are left to simplify the first term. As in the proof of Proposition 5.3 the sum on $k_1,k_2$ equals 
	\begin{equation*}
	2\sum_{\substack{k_1\pmod{p}\\k_2\pmod{6p}\\k_1\equiv k_2\pmod{2}}}\left(2\zeta_{2p}^{k_1+k_2}+\zeta_p^{k_1}\right)\mathcal J_{\boldsymbol{k}}\left(-\frac{1}{\tau}\right).
	\end{equation*}
	In the contribution from the second term, we change $k_1$ into $\frac{k_1+k_2}{2}$ and $k_2$ into $\frac{k_2-3k_1}{2}$ giving the claim. 
\end{proof}

We next write $\mathcal E_2$ in terms of the $\mathcal K$-functions.
\begin{lemma}
	We have
	\begin{equation}\label{E2K}
	\mathcal E_2(\tau)=\mathcal K_{(1,3)}(\tau).
	\end{equation}
\end{lemma}

Proposition 5.5 yields the following transformation for $\mathcal E_2$.
\begin{corollary}\label{prop2}
	We have 
	\begin{align*}
	\mathcal E_2(\tau) &= \frac{i}{8\pi p(-i\tau)^2} \sum_{\substack{k_1\pmod{p}\\k_2\pmod{6p}\\k_1\equiv k_2\pmod{2}}} \zeta_{2p}^{k_1+k_2} \mathcal K_{\boldsymbol{k}}\left(-\frac{1}{\tau}\right) + \frac{i}{4} \sum_{\boldsymbol{\alpha}\in\mathscr S^*}  H_{2,\boldsymbol{\alpha}}(\tau)\\
	&\quad-\frac{\sqrt{3}}{8\pi}\sum_{\delta\in\{0,1\}}\Big(2\left(I_{\Theta_1(2p,1+p\delta,2p;\cdot)}(\tau)-r_{\Theta_1(2p,1+p\delta,2p;\cdot)}(\tau)\right)r_{\Theta_0(6p,3+3p\delta,6p;\cdot)}(\tau)\\
	&\qquad\qquad\qquad\qquad-\left(I_{\Theta_1(2p,2+p\delta,2p;\cdot)}(\tau)-r_{\Theta_1(2p,1+p\delta,2p;\cdot)}(\tau)\right)r_{\Theta_0(6p,3p\delta,6p;\cdot)}(\tau)\Big).
	\end{align*}
\end{corollary}
\begin{proof}
	The claim follows from Proposition \ref{propGenTrans}. Reversing the calculations required for the proof of \eqref{E2K} yields that the second summand equals $\frac{i}{4}\sum_{\boldsymbol{\alpha}\in\mathscr S^*}  H_{2,\boldsymbol{\alpha}}(\tau)$.
\end{proof}

\subsection{Proof Theorem \ref{maintheorem1}}
We are now ready to prove a refined version of Theorem \ref{maintheorem1}.
\begin{theorem}\label{quantumref}

	\begin{enumerate}[leftmargin=*]
		
		\item[\textnormal{(1)}] The function $\widehat{F}_1:\Q\to\C$ defined by $\widehat{F}_1(\frac{h}{k}):=F_1(e^{2\pi i\frac{ph}{k}})$ is a component of a  vector-valued quantum modular form of depth two and weight one.
		
		\item[\textnormal{(2)}] The function $\widehat{F}_2:\mathcal \Q \to \C$ defined by $\widehat{F}_2(\frac{h}{k}):=F_2(e^{2\pi i\frac{ph}{k}})$ is a component of a vector-valued quantum modular form of depth two and weight two.
		
	\end{enumerate}
\end{theorem}

\begin{proof}
	(1) We have, by \eqref{AsE1},
	$$
	\widehat{F}_1\left(\frac{h}{k}\right) = \lim_{t\to 0^+} F_1\left(e^{2\pi i\frac{ph}{k}-t}\right)=A_{hp_1,\frac{k}{p_2}}(0) = \lim_{t\to 0^+}\mathbb E_1\left(\frac{it}{2\pi}-\frac{h}{k}\right),
	$$
	where $p_1:= p/\gcd(k,p)$, $p_2:=\gcd(k,p)$. Corollary \ref{propE1} and Proposition \ref{propJ} then give the claim.
	
	(2) The relation \eqref{AsE2} gives 		$$
	\widehat{F}_2\left(\frac{h}{k}\right)= \lim_{t\to 0^+} {F}_2\left(e^{2\pi i\frac{ph}{k}-t}\right) = B_{hp_1,\frac{k}{p_2}}(0) = \lim_{t\to0^+}\mathbb E_2\left(\frac{it}{2\pi}-\frac{h}{k}\right).
	$$
	Corollary 5.4 and Proposition \ref{propGenTrans} then yields the claim.
\end{proof}

\section{Higher Mordell integrals}

\subsection{Proof of Theorem \ref{theoremHigherMordell}}	

	We first assume that $\alpha_j\not \in \mathbb Z$. Via analytic continuation, it is enough to show the theorem for $\tau=iv$.
	We first claim that
	\begin{equation}\label{M2sum2}
	H_{1,\boldsymbol{\alpha}}(iv)=
	2\lim_{r \to \infty}\sum_{\substack{\boldsymbol{n}\in\boldsymbol{\alpha}+\Z^2\\ \lvert n_j-\alpha_j\rvert\leq r}}M_2\left(\sqrt{3}; \sqrt{\frac{v}{2}}\left(\sqrt{3}\left(2n_1+n_2\right), n_2\right)\right)e^{2\pi Q(\boldsymbol{n})v}.
	\end{equation}
    For this we write (which follows from shifting in (6.1) of \cite{BKM} $w_j\mapsto2iw_j-\overline \tau$)
	\begin{align*}
	e^{4\pi Q(\boldsymbol{n})v}M_2&\left(\sqrt{3};\sqrt{3v}(2n_1+n_2),\sqrt{v}n_2\right)\\
	&= \sqrt{3} \left(2n_1+n_2\right) n_2\int_0^\infty \frac{e^{-3\pi\left(2n_1+n_2\right)^2 w_1}}{\sqrt{w_1+v}}\int_{w_1}^\infty \frac{e^{-\pi n_2^2w_2}}{\sqrt{w_2+v}}dw_2dw_1\\
	&\quad+\sqrt{3} \left(3n_1+2n_2\right) n_1\int_0^\infty \frac{e^{-\pi\left(3n_1+2n_2\right)^2 w_1}}{\sqrt{w_1+v}}\int_{w_1}^\infty \frac{e^{-3\pi n_1^2w_2}}{\sqrt{w_2+v}}dw_2dw_1. 
	\end{align*}
	Then we change $v\mapsto \frac{v}{2}$, sum over those $\boldsymbol{n}\in\boldsymbol{\alpha}+\Z^2$ satisfying $\lvert n_j-\alpha_j\rvert \leq r$ and let $r\to\infty$. On the right-hand side we may use Lebesgue's dominated convergence theorem and can reorder the absolutely converging series inside the integral to obtain \eqref{M2sum2}.
	
To finish the proof, we rewrite \eqref{defineM2}, to obtain (assuming $N_2,N_1-\kappa N_2\neq0$) 
\begin{equation}\label{MordN}
M_2\left(\kappa; \sqrt{v}\boldsymbol{N}\right)
=-\frac1{\pi^2} e^{-\pi v\left(N_1^2+N_2^2\right)}\int_{\R^2}
\frac{e^{-\pi vw_1^2-\pi vw_2^2}}{\left(w_2-i N_2\right)\left(w_1-\kappa w_2-i\left(N_1-\kappa N_2\right)\right)}\boldsymbol{dw}.
\end{equation}
Thus in particular (for $n_1,n_2\neq 0$)
\begin{align*}
M_2\left(\sqrt{3}; \sqrt{\frac{v}{2}}\left(\sqrt{3}\left(2n_1+n_2\right), n_2\right)\right)&=-\frac1{\pi^2}e^{-2\pi Q(\boldsymbol{n})v}
\int_{\R^2}\frac{e^{-\frac{\pi vw_1^2}{2}-\frac{\pi vw_2^2}{2}}}{\left(w_2-in_2\right)\left(w_1-\sqrt{3}w_2-2\sqrt{3} in_1\right)}\boldsymbol{dw}\notag \\
&=-\frac1{\pi^2}e^{-2\pi Q(\boldsymbol{n})v}
\int_{\R^2}\frac{e^{-\frac{3\pi v\left(2w_1+w_2\right)^2}{2}-\frac{\pi vw_2^2}{2}}}{\left(w_2-in_2\right)\left(w_1- in_1\right)}\boldsymbol{dw},
\end{align*}
making the change of variables $w_1\mapsto 2\sqrt{3} w_1+\sqrt{3}w_2$. This implies that
\begin{multline}\label{sumM2}
\lim_{r \to \infty}\sum_{\substack{\boldsymbol{n}\in\boldsymbol{\alpha}+\Z^2\\ \lvert n_j-\alpha_j\rvert\leq r}}M_2\left(\sqrt{3}; \sqrt{\frac{v}{2}}\left(\sqrt{3}\left(2n_1+n_2\right), n_2\right)\right)e^{2\pi Q(\boldsymbol{n})v}\\
=-\frac1{\pi^2}\lim_{r \to \infty}\sum_{\substack{\boldsymbol{n}\in\boldsymbol{\alpha}+\Z^2\\ \lvert n_j-\alpha_j\rvert\leq r}}\int_{\R^2}
\frac{e^{-2\pi vQ(\boldsymbol{w})}}{\left(w_2-i n_2\right)\left(w_1-in_1\right)}\boldsymbol{dw}.
\end{multline}
Using 
\begin{equation*}
\pi \cot(\pi x) = \lim_{r \to \infty}\sum_{k=-r}^r \frac{1}{x+k},
\end{equation*}
we obtain that the sum over the integrand (without the exponential factor) is
\begin{align*}
&
-\lim_{r \to \infty}\sum_{\substack{\boldsymbol{n}\in\Z^2\\ \lvert n_j\rvert\leq r}}
\left(\frac{1}{iw_1+\alpha_1+n_1}\right)
\left(\frac{1}{iw_2+ \alpha_2+n_2}\right)
=
-\pi^2 \cot\left(\pi\left(iw_1+\alpha_1\right) \right)\cot\left(\pi\left(iw_2+ \alpha_2\right)\right)
.
\end{align*}
Using again Lebesgue's theorem of dominated convergence, one can show that one can interchange the limit and the integration in \eqref{sumM2} to obtain
	\begin{align*}
	H_{1,\boldsymbol{\alpha}}(\tau)=\int_{\R^2}
	\cot\left(\pi iw_1+\pi\alpha_1\right)
	\cot\left(\pi iw_2+\pi\alpha_2\right)e^{2\pi i\tau Q(\boldsymbol{w})}
	\boldsymbol{dw}.
	\end{align*}	
Using
\begin{equation*}
\cot(x+iy) = -\frac{\sin(2x)}{\cos(2x)-\cosh(2y)}+i\frac{\sinh(2y)}{\cos(2x)-\cosh(2y)}.
\end{equation*}
then yields, 
\begin{align}\label{Hint}
H_{1,\boldsymbol{\alpha}}(\tau)= 2\int_{\R^2} \left(\mathcal G_{\alpha_1}(w_1) \mathcal G_{\alpha_2}(w_2)-\mathcal F_{\alpha_1}(w_1)\mathcal F_{\alpha_2}(w_2)\right)e^{2\pi i\tau  Q(\boldsymbol{w})}\boldsymbol{dw}.
\end{align}
This gives the claim of Theorem \ref{theoremHigherMordell} in this case.

We next turn to the case that $\alpha_j\in\Z$ for exactly one $j\in\{1,2\}$. We only consider the case $\alpha_1\in\mathbb Z$, since the case $\alpha_2\in\Z$ goes analogously. Since the integrand in $H_{1,\boldsymbol{\alpha}}$ is invariant under $\alpha_j\mapsto \alpha_j+1$, we may assume that $\alpha_1=0$. One directly sees from \eqref{Hint} that in this case 
\[
H_{1,(0,\alpha_2)}(\tau)= -2\lim_{\alpha_1\to 0}\int_{\R^2}\mathcal F_{\alpha_1}(w_1)  \mathcal F_{\alpha_2}(w_2)e^{2\pi i\tau Q(\boldsymbol{w})}\boldsymbol{dw}.
\]
Using that $\mathcal F_0(-w_1)=-\mathcal F_0(w_1)$, we obtain
\begin{equation}\label{HO}
H_{1,(0,\alpha_2)}(\tau)= -\int_{\R^2}\mathcal F_0(w_1)  \mathcal F_{\alpha_2}(w_2)e^{2\pi i\tau\left(3w_1^2+w_2^2\right)}\sum_{\pm}\pm e^{\pm 6\pi i\tau w_1w_2}\boldsymbol{dw}.
\end{equation}
Now write
\[
\mathcal F_0(w_1)=\left(\mathcal F_0(w_1)-\frac{1}{\pi w_1}\right)+\frac{1}{\pi w_1}.
\]
The contribution of the first term to the integral now exists and gives, changing $w_1\mapsto -w_1$ for the minus sign
\begin{align*}
-2\int_{\R^2}\left(\mathcal F_0(w_1)-\frac{1}{\pi w_1}\right)  \mathcal F_{\alpha_2}(w_2)e^{2\pi i\tau Q(\boldsymbol{w})} \boldsymbol{dw}.
\end{align*}
For the second term we write
\begin{align}\label{Fdec}
\mathcal F_{\alpha_2}(w_2)&=\left(\mathcal F_{\alpha_2}(w_2)-\mathcal F_{\alpha_2}\left(w_2\pm \frac{3w_1}{2}\right)\right)+F_{\alpha_2}\left(w_2\pm \frac{3w_1}{2}\right).
\end{align}
The first term in \eqref{Fdec} contributes to \eqref{HO}, changing $w_1\mapsto -w_1$ for the minus sign
\begin{align*}
-\frac{2}{\pi} \int_{\R^2}w_1^{-1}\left(\mathcal F_{\alpha_2}(w_2)-\mathcal F_{\alpha_2}\left(w_2+  \frac{3w_1}{2}\right)\right) e^{2\pi i\tau Q(\boldsymbol{w})}\boldsymbol{dw}.
\end{align*}
For the final term in \eqref{Fdec} we use that $3w_1^2+w_2^2\pm 3w_1w_2=(w_2\pm\frac{3w_1}{2})^2+\frac34 w_1^2$, to obtain
\begin{align*}
-\frac{1}{\pi} \int_{\R}\frac{e^{\frac{3\pi i \tau w_1^2}{2}}}{w_1}\int_{\R} \sum_{\pm}\pm \mathcal{F}_{\alpha_2}\left(w_2\pm \frac{3w_1}{2}\right) e^{2\pi i  \tau\left(w_2\pm \frac{3w_1}{2}\right)^2}dw_2dw_1.
\end{align*}
The inner integral on $w_2$ now vanishes,
which may be seen by changing in the integral on $w_2$ for the minus sign $w_2\mapsto w_2+3w_1$. Combining, the theorem statement follows.

\subsection{Proof of Theorem \ref{diffTh}}
\begin{proof}[Proof of Theorem \ref{diffTh}]
From (6.3) and (6.4) of \cite{BKM}, one obtains that
\begin{align*}
&\frac{1}{2\pi i}\left[\frac{\partial}{\partial z} \left(M_2\left(\sqrt{3};\sqrt{3v}(2n_1+n_2),\sqrt{v}\left(n_2-\frac{2\im(z)}{v}\right)\right)e^{2\pi in_2z}\right)\right]_{z=0} e^{4\pi v Q(\boldsymbol{n}) }\\
&= -\frac{\sqrt{3}}{2\pi} (2n_1+n_2) \int_{0}^\infty \frac{e^{-\frac{3\pi}{2}(2n_1+n_2)^2w_1}}{\sqrt{w_1+2v}} \int_{w_1}^\infty \frac{e^{-\frac{\pi}{2}n_2^2w_2}}{(w_2+2v)^{\frac32}} dw_2dw_1\\
&\quad +\frac{\sqrt{3}}{4\pi} (3n_1+2n_2) \int_0^\infty \frac{e^{-\frac{\pi}{2}(3n_1+2n_2)^2w_1}}{\sqrt{w_1+2v}} \int_{w_1}^\infty \frac{e^{-\frac{3\pi n_1^2w_2}{2}}}{(w_2+2v)^{\frac32}}dw_2dw_1\\
&\quad - \frac{\sqrt{3}}{4\pi} n_1 \int_0^\infty \frac{e^{-\frac{\pi}{2}(3n_1+2n_2)^2 w_1}}{(w_1+2v)^{\frac32}} \int_{w_1}^\infty \frac{e^{-\frac{3\pi n_1^2w_2}{2}}}{\sqrt{w_2+2v}}dw_2dw_1.
\end{align*}
Then we sum over $\boldsymbol{n}\in\boldsymbol{\alpha}+\Z^2$ satisfying $\lvert n_j-\alpha_j\rvert \leq r$ and let $r\to\infty$. On the right hand side we use Lebesgue's dominated convergence theorem and can reorder the absolutely converging series inside the integral to obtain
\begin{multline*}
\frac{1}{2\pi i} \lim_{r \to \infty}\sum_{\substack{\boldsymbol{n}\in\boldsymbol{\alpha}+\Z^2\\ \lvert n_j-\alpha_j\rvert\leq r}} \left[\frac{\partial}{\partial z} \left(M_2\left(\sqrt{3};\sqrt{3v}(2n_1+n_2),\sqrt{v}\left(n_2-\frac{2\im(z)}{v}\right)\right) e^{2\pi in_2z}\right)\right]_{z=0} e^{4\pi v Q(\boldsymbol{n}) }\\=\frac{1}{2i} H_{2,\boldsymbol{\alpha}}(2iv).
\end{multline*}
We now use \eqref{MordN} and change variables $w_1\mapsto 2 \sqrt{3}w_1+\sqrt{3}w_2$, to obtain
\begin{align*}
&M_2\left(\sqrt{3};\sqrt{3v}(2n_1+n_2),\sqrt{v}\left(n_2-\frac{2\im(z)}{v}\right)\right) e^{2\pi in_2z} \\
&= -\frac{1}{\pi^2} e^{-\pi v\left(3(2n_1+n_2)^2 + \left(n_2-\frac{2\im(z)}{v}\right)^2\right)+2\pi in_2 z} \\
&\hspace{2cm}\times\int_{\R^2} \frac{e^{-\pi v w_1^2-\pi v w_2^2}}{\left(w_2-i\left(n_2-\frac{2\im(z)}{v}\right)\right)\left(w_1-\sqrt{3}w_2-i\left(2\sqrt{3}n_1+2\sqrt{3}\frac{\im(z)}{v}\right)\right)}\boldsymbol{dw}\\
&= -\frac{1}{\pi^2} e^{-\pi v\left(3(2n_1+n_2)^2 + \left(n_2-\frac{2\im(z)}{v}\right)^2\right)+2\pi in_2 z}\int_{\R^2} \frac{e^{-4\pi v Q(\boldsymbol{w})}}{ \left(w_2-i\left(n_2-\frac{2\im(z)}{v}\right)\right)\left(w_1-i\left(n_1+\frac{\im(z)}{v}\right)\right)}\boldsymbol{dw}.
\end{align*}
Thus
\begin{align*}
&\left[\frac{\partial}{\partial z}M_2\left(\sqrt{3};\sqrt{3v}(2n_1+n_2),\sqrt{v}\left(n_2-\frac{2\im(z)}{v}\right)\right) e^{2\pi in_2z}\right]_{z=0}e^{4\pi v Q(\boldsymbol{n})} \\
&\qquad= \frac{2}{\pi} \int_{\R^2} \frac{w_2e^{-4\pi vQ(\boldsymbol{w})}}{(w_2-in_2)\left(w_1-in_1\right)}\boldsymbol{dw}.
\end{align*}
Exactly as in the proof of Theorem \ref{theoremHigherMordell}, we then obtain 
\begin{multline*}
\lim_{r \to \infty}\sum_{\substack{\boldsymbol{n}\in\boldsymbol{\alpha}+\Z^2\\ \lvert n_j-\alpha_j\rvert\leq r}} \left[\frac{\partial}{\partial z}M_2\left(\sqrt{3};\sqrt{3v}(2n_1+n_2),\sqrt{v}\left(n_2-\frac{2\im(z)}{v}\right)\right) e^{2\pi in_2z}\right]_{z=0}e^{4\pi vQ(\boldsymbol{n})}\\
=\frac{\pi}{\sqrt{3}} \int_{\R^2} w_2 e^{-4\pi vQ(\boldsymbol{w})} \cot\left(\pi iw_2+\pi \alpha_2\right) \cot\left(\pi i w_1+\pi \alpha_1\right) \boldsymbol{dw}.
\end{multline*}
Observing that on the right hand side the integral over the real part vanishes, gives
\begin{align*}
H_{2,\boldsymbol{\alpha}}(\tau)=-2i\int_{\mathbb R^2}w_2\left(\mathcal G_{\alpha_1}(w_1)\mathcal F_{\alpha_2}(w_2)+\mathcal F_{\alpha_1}(w_1)\mathcal G_{\alpha_2}(w_2)\right) e^{2\pi i\tau Q(\boldsymbol{w})}\boldsymbol{dw}.
\end{align*}
The case $\alpha_1\not\in\mathbb Z$ follows directly.

For $\alpha_1\in\mathbb Z$, we obtain
\[
H_{2,\boldsymbol{\alpha}}(\tau)=-2i\int_{\mathbb R^2} \mathcal F_{0}(w_1)\mathcal G_{\alpha_2}^*(w_2)e^{2\pi i\tau Q(\boldsymbol{w})}\boldsymbol{dw}.
\]
Now the claim follows as in the proof of Theorem \ref{theoremHigherMordell}.
\end{proof}

\section{Future work}

Here we discuss a few future directions. We also announce a result that will appear in full detail in our forthcoming work \cite{BKM2}.

\subsection{Further examples of rank two false theta functions}
In addition to the function $F$ studied in \cite{BKM}, there are additional rank two false theta functions studied by the first and third author in \cite{BM2}.
To be more precise, define, for $1\leq s_1,s_2\leq p$
\begin{align*}
\mathbb F_{s_1,s_2}(q) &:= \sum_{\substack{m_1,m_2\geq 1 \\ m_1\equiv m_2\pmod{3}}}\min(m_1,m_2) q^{\frac{p}{3}\left(\left(m_1-\frac{s_1}{p}\right)^2 + \left(m_2-\frac{s_2}{p}\right)^2 + \left(m_1-\frac{s_1}{p}\right)\left(m_2-\frac{s_2}{p}\right)\right)} \\
&\qquad\qquad\times\left(1-q^{m_1s_1}-q^{m_2s_2} + q^{m_1s_1+(m_1+m_2)s_2}+q^{m_2s_2+(m_1+m_2)s_1}- q^{(m_1+m_2)(s_1+s_2)}\right).
\end{align*}
We will show in \cite{BKM2} that these series are also {higher depth} quantum modular forms with quantum set $\mathbb{Q}$. We believe that these series decompose 
into two vector-valued higher depth  quantum modular forms of weight one and two.


\subsection{Example: two-dimensional vector-valued quantum modular forms of depth two}

The previous  two-parametric family of rank two false theta functions $\mathbb{F}_{s_1,s_2}(q)$ takes a particularly nice shape for $p=2$.
In this case it can be shown that the only contribution comes from the weight one component and that only two false theta functions contribute. Their companions are double 
Eichler integrals with a basis
\begin{align*}
& \int_{-\overline{\tau}}^{i\infty}\int_{w_1}^{i\infty} \frac{\eta(w_1)^3  \eta(3w_2)^3 }{\sqrt{-i(w_1+\tau)} \sqrt{-i(w_2+\tau)}}dw_2dw_1 \qquad  \int_{-\overline{\tau}}^{i\infty}\int_{w_1}^{i\infty} \frac{\eta(w_1)^3 \eta\left(\frac{w_2}{3}\right)^3 }{\sqrt{-i(w_1+\tau)} \sqrt{-i(w_2+\tau)}}dw_2dw_1.
\end{align*}
This gives a two-dimensional vector-valued quantum modular form of depth two and weight one.



\end{document}